\documentclass{article}
\usepackage[pdftex]{hyperref}
\hypersetup{
  colorlinks   = true, 
  urlcolor     = gray, 
  linkcolor    = red, 
  citecolor   =  blue
}
\usepackage{amsmath,amsfonts,amsthm,amssymb,graphicx,xcolor}
\usepackage[all,2cell,ps]{xy}

\theoremstyle{plain}
\newtheorem{theorem}[equation]{Theorem}
\newtheorem{corollary}[equation]{Corollary}
\newtheorem{lemma}[equation]{Lemma}
\newtheorem{proposition}[equation]{Proposition}

\theoremstyle{definition}
\newtheorem{define}[equation]{Definition}
\newtheorem{definitions}[equation]{Definition}

\newtheorem{remark}[equation]{Remark}

\newtheorem{question}[equation]{Question}

\newcommand{\p}{\partial}

\newcommand{\IH}{\mathbb{H}}

\newcommand{\IN}{\mathbb{N}}

\newcommand{\IR}{\mathbb{R}}

\newcommand{\inj}{\mathrm{inj}}

\title{Harmonic Forms, Price Inequalities, and Benjamini-Schramm Convergence}
\author{\small{Luca F. Di Cerbo}\footnote{Partially supported by NSF grant DMS-2104662} \\ \scriptsize{University of Florida}\\ \footnotesize{\textsf{ldicerbo@ufl.edu}} \and \small{Mark Stern}\footnote{Partially supported by Simons Foundation Grant 3553857} \\ \scriptsize{Duke University} \\ \footnotesize{\textsf{stern@math.duke.edu}}}
\date{}

\begin{document}

\maketitle

\begin{abstract}
We study Betti numbers of sequences of Riemannian manifolds which Benjamini-Schramm converge to their universal covers. Using the Price inequalities we developed elsewhere, we derive two distinct convergence results. First, under a negative Ricci curvature assumption and \emph{no} assumption on sign of the sectional curvature, we have a convergence result for \emph{weakly} uniform discrete sequences of closed Riemannian manifolds. In the negative sectional curvature case, we are able to remove the weakly uniform discreteness assumption. This is achieved by combining a refined Thick-Thin decomposition together with a Moser iteration argument for harmonic forms on manifolds with boundary.
\end{abstract}

\vspace{8cm}

\tableofcontents

\vspace{1cm}


\section{Introduction}
Let $(M^n, g)$ be a closed Riemannian manifold.  Define the normalized Betti numbers and $L^2$-Betti numbers respectively as:
\begin{align}\tilde b_{k,g}(M):= \frac{b_k(M)}{Vol(M)}\,\,\text{ and   }\,\,\tilde b_{k,g}^{(2)}(M):= \frac{b_k^{(2)}(M)}{Vol(M)},
\end{align}
where $b_k(M)$ denotes the $kth$ Betti number of $M$ and $b_k^{(2)}(M)$ denotes the $kth$ $L^2$-Betti number.  
In  an influential paper \cite{Luck}, L\"uck shows that if $M$ is a closed manifold  with residually finite fundamental group, then 
\begin{align}\label{luck}\tilde b_{k,g}^{(2)}(M)=\lim_{l\to\infty}\tilde b_{k,g}(M_l),
\end{align} for any  tower of coverings $\{M_l\}_l$ of $M$  associated to a cofinal filtration of its fundamental group. The $L^2$-Betti numbers were originally defined analytically by Atiyah in \cite{Atiyah}, and L\"uck's theorem provides a remarkable connection between analysis and topology which has inspired considerable mathematics  in the last two or three decades, see for example the bibliography of  \cite{LuckBook}. 

More recently, Abert \emph{et al.} in \cite{Bergeron} and \cite{Bergeron2}  generalized Luck's approximation theorem in the context of lattices in Lie groups and in the context of finite volume manifolds of negative curvature. To describe this generalization, we first recall a rather weak notion of convergence of Riemannian manifolds,  Benjamini-Schramm convergence, which is adapted from graph theory \cite{BS01}. In Riemannian terms, this convergence is given as follows.

\begin{definitions}\label{BSdef}
	Let $(M_l, g_l)_l$ be a sequence of closed Riemannian manifolds which share a common universal Riemannian cover $(X, g)$. Given $x\in M_l$, we denote by $\inj_{g_{l}}(x)$ the injectivity radius of $(M_l, g_l)$ at $x$. We define the $R$-thin part of $(M_l, g_l)$, denoted  $(M_l)_{<R}$, by
	\[
	(M_l)_{<R}:=\{x\in M_{l}\, | \, \inj_{g_{l}}(x)< R\}.
	\]
Define a relative measure of the thin regions of $M_l$ by 
\begin{align}\rho(M_l,R):= \frac{Vol_{g_l}((M_l)_{<R})}{Vol_{g_l}(M_{l})}.\end{align}
We say that the sequence $(M_l, g_l)_l$ Benjamini-Schramm converges to $(X, g)$, if for any $R>0$ we have
	\[
	\lim_{l\to\infty}\rho(M_l,R)=0.
	\]
		Finally, we say that 
the sequence $(M_l, g_l)$ is \emph{uniformly discrete}, if there exists $\epsilon>0$ such that for any $l\in\IN$:
	\[
	\min_{x\in M_l}\inj_{g_l}(x)\geq \epsilon.
	\]
\end{definitions}
\vspace{0.1 in}
We can now state one of the main results in \cite{Bergeron}.\\

\begin{theorem}[Corollary 1.4 in \cite{Bergeron}]\label{luckbs}
Let $\{\Gamma_l\}_l$ be a sequence of uniformly discrete, torsion free lattices acting co-compactly on a symmetric space $G/K$ of non-compact type. Let $\{\Gamma_l\backslash G/K\}_l$ be the associated sequence of compact locally symmetric spaces. For any $k\leq\dim(G/K)$, if $\{\Gamma_l\backslash G/K\}_l$ Benjamini-Schramm converges to $G/K$ (equipped with the standard symmetric metric), we have 
\[
\lim_{l\to\infty}\tilde b_{k,g}(\Gamma_l\backslash G/K) =\beta^{(2)}_{k}(G/K),
\]	
where the $k$-th $L^2$-Betti number of the symmetric space, $\beta^{(2)}_{k}(G/K)$, is defined in \cite[6.24]{Bergeron}, and satisfies $\tilde b_{k,g}^{(2)}(\Gamma\backslash G/K)=  \beta^{(2)}_{k}(G/K)$ for every cocompact torsion free $\Gamma$. 
\end{theorem}

When  $\{\Gamma_l\}_l$ is a cofinal filtration of a given torsion free lattice acting co-compactly on $G/K$ (\emph{cf.} Theorem 2.1 in \cite{Wallach}),  the sequence of coverings $\{\Gamma_l\backslash G/K\}_l$ Benjamini-Schramm converges to $G/K$, and Theorem \ref{luckbs} is a genuine generalization of  L\"uck's original approximation theorem in the case of locally symmetric spaces. We also observe that, for subgroups of a fixed lattice, Theorem 1.12 in  \cite{Bergeron} provides effective bounds on the normalized Betti numbers. 

In the real hyperbolic case $G/K=\IH^n$, Abert \emph{et al.} in \cite{Bergeron} obtain their strongest result. Remarkably, they are able to remove the uniform discreteness assumption on the lattices. 

\begin{theorem}[Theorem 1.8. in \cite{Bergeron}]\label{luckbs2}
	Let $\{\Gamma_l\backslash\IH^n\}_l$ be a sequence of compact hyperbolic manifolds of dimension $n$ that Benjamini-Schramm converge to $\IH^n$. For any $k\leq n$, we have
	\[
	\lim_{l\to\infty}\tilde b_k(\Gamma_l\backslash \IH^n) =\beta^{(2)}_{k}(\IH^n).
	\]
\end{theorem}

More recently, in the preprint \cite{Bergeron2}, four of the seven authors of \cite{Bergeron}, extended L\"uck's approximation theorem to sequences of pinched negatively curved manifolds which Benjamini-Schramm converge to their universal cover.\\ 

 
In this paper, we contribute to this circle of ideas by extending the techniques of \cite{DS17} to prove and {\em quantify} vanishing of normalized Betti numbers  (in certain degrees) along sequences of closed Riemannian manifolds which Benjamini-Schramm converge to their universal covers. Whereas the focus of \cite{Bergeron} and \cite{Bergeron2} is to {\em relate} normalized Betti numbers and $L^2$-Betti numbers, in this paper we  concentrate on providing bounds and vanishing results for these quantities. Here we consider geometries more general than those considered in \cite{Bergeron2}.  Our techniques are rather distant from those of \cite{Bergeron} and \cite{Bergeron2}. Indeed, we rely on geometric inequalities for harmonic forms on negatively curved Riemannian manifolds which we described in \cite{DS17}. In particular, some of our results do not require \emph{any} direct assumption on the sectional curvature. 

The next definition is tailored to our analytical techniques, and it will be used throughout this paper. This definition is related to the notions of convergence considered in \cite{Bergeron} and \cite{Bergeron2}, but at the same time it contains some new elements.  

\begin{define}\label{weakdisc}
Let $(M_l, g_l)_l$ be a sequence of closed Riemannian manifolds which share a common universal Riemannian cover $(X, g)$. We say a sequence of manifolds $(M_l,g_l)_l$ is {\em weakly uniformly discrete} and converges to $(X, g)$ if there exists a sequence $\{R_l\}_l\subset (0,\infty)$ with 
\[
\lim_{l\to\infty}R_l = \infty,
\]
such that 
\[
\lim_{l\to\infty}\Big(1+\frac{1}{\inj_{g_l}(M_l)}\Big)^n\rho(M_l,R_l)=0.
\]
\end{define}

Before listing our main results, it is important to state the precise connection between the notion of convergence given in Definition \ref{weakdisc} and the usual Benjamini-Schramm convergence (\emph{cf.} \cite{Bergeron} and \cite{Bergeron2}).

\begin{remark}
	If a sequence of manifolds $(M_l,g_l)_l$  Benjamini-Schramm converges, then there is always a sequence $\{R_l\}_l$ converging to $\infty$  such that 
	\begin{align}\label{newdef2}
	\lim_{l\to\infty}\rho(M_l, R_l)=0.
	\end{align}
	Hence every uniformly discrete sequence which Benjamini Schramm converges is weakly uniformly discrete. 
On the other hand, weakly uniformly discrete sequences may well have injectivity radius that goes to zero along a subsequence, and therefore need not be uniformly discrete in the sense of Definition \ref{BSdef}.
\end{remark}

We can now state our first result which requires only a negative {\em Ricci} curvature assumption, and no uniform lower bound on the injectivity radius. On the other hand, we require the weakly uniformly discrete assumption (\emph{cf.} Definition \ref{weakdisc}). 

\begin{theorem}\label{RBS}
	Let $(X^n, g)$ be a simply connected manifold without conjugate points and with $-1\leq\sec_g\leq 1$. Assume there exists $\delta>0$ such that
	\[
	-Ric_g\geq \delta g.
	\]
	Let $(M_l, g_l)$ be a weakly uniformly discrete sequence of closed manifolds converging to $(X, g)$. Then for any $k\in \IN$ such that $\delta>4k^2$, we have
	\[
	\lim_{l\to\infty}\tilde b_{k,g}(M_l) =0.
	\]
\end{theorem}

We remark that, in \cite[Theorem 122]{DS17}, under the same curvature assumptions as in Theorem \ref{RBS}, we proved the following vanishing theorem for $L^2$-Betti numbers:
\[
b^{(2)}_{k}(M_l)=0,\quad \text{for all}\quad l.
\] 
Thus, Theorem \ref{RBS} asserts the convergence along the Benjamini-Schramm sequence of certain normalized Betti numbers to the corresponding $L^2$-Betti number.  \\

If we assume the sectional curvature to be strictly negative, the techniques developed in \cite{DS17} cover a larger range of Betti numbers. Using this fact, we are able to prove the following (see Theorem \ref{barvolume} for a stronger statement and details of the proof).

\begin{theorem}\label{aBS}
	Let $(X^n, g)$ be a simply connected manifold with 
	\[
	-a^2\leq\sec_g\leq -1,
	\]
	and $a\geq 1$. Let $(M_{l}, g_{l})$ a sequence of closed Riemannian manifolds BS-converging to $(X^n, g)$. For any $k\in\IN$ such that 
	\[
	a_{n, k}:=(n-1)-2ka\geq 0,
	\]
	we have
	\[
	\lim_{l\rightarrow\infty}\tilde b_{k,g}(M_{l})  =0.
	\]
\end{theorem}
\vspace{0.1in}

Once again, under the same curvature assumptions as in Theorem \ref{aBS}, we have elsewhere proved the following vanishing theorem for $L^2$-Betti numbers
\[
b^{(2)}_{k}(X^n)=0 \quad \Rightarrow \quad b^{(2)}_{k}(M_l)=0 \quad \text{for all} \quad l;
\] 
see Section 7 and Proposition 126 in \cite{DS17} (\emph{cf.} also Proposition 4.1 in \cite{Donnelly} when $a_{n, k}>0$, and \cite{Dodziuk} when $a=1$). Hence  Theorem \ref{aBS} is already a consequence of the preceding references and \cite{Bergeron2}.  None the less, our convergence result is completely independent of the theory of $L^2$-Betti numbers, and it follows directly from the Price inequalities for harmonic forms we developed in \cite{DS17}.  Indeed all of the analysis can be performed directly on the sequence of compact manifolds, without the need of studying $L^2$-harmonic forms on the universal Riemannian cover.\\

Observe that Theorem \ref{aBS}, unlike Theorem \ref{RBS}, does \emph{not} require any uniform discreteness assumption. This greater generality is present in \cite{Bergeron2} and Theorem \ref{luckbs2} as well, and it depends crucially on the fact that in the negative sectional curvature regime, thanks to the Gromov-Margulis lemma (\emph{cf.}  \cite{BGS85}), we understand quite well the topology of regions with small injectivity radius. On the other hand, our proof is substantially different from the approach presented in \cite{Bergeron2}. \\ \\

\noindent\textbf{Acknowledgments}. LFDC thanks Duke University, which he visited on two occasions during the preparation of this work, for its hospitality and support. The authors thank Xiaolong Hans Han for his interest in this work and for pointing out an inaccuracy in a preliminary version of this work. Finally, we thank the referee for suggestions that improved the exposition.

\section{Dimension Estimates Revisited}

Let $(M^n, g)$ be a closed Riemannian manifold, and denote by $\mathcal{H}^k_{g}(M)$ the finite dimensional vector space of harmonic $k$-forms. Define normalized Betti numbers  
\begin{align}\tilde b_{k,g}(M):= \frac{b_k(M)}{Vol_{g}(M)}.
\end{align}
In Section 5 of \cite{DS17} and again in Lemma \ref{keyestimate} below, we show 
\begin{align}\label{peakest}
\tilde b_{k,g}(M)\leq \left(\begin{array}{c}n\\k\end{array}\right)\max \{\frac{\|\alpha\|^2_{L^\infty}}{\|\alpha\|^2_{L^2}}:\alpha\in\mathcal{H}^k_g(M)\setminus\{0\}\}.
\end{align}
 Under various hypotheses on  the Ricci curvature or the Riemannian curvature and $k$, in \cite{DS17} we showed exponential or polynomial bounds in the injectivity radius for the normalized Betti numbers.  Those estimates, in conjunction with \eqref{peakest} suffice to establish convergence to zero of sequences of  normalized Betti numbers of closed Riemannian manifolds whose injectivity radii diverge, for example, sequences of real hyperbolic manifolds associated to a cofinal filtration of a given torsion free co-compact lattice in $\textrm{Iso}(\IH^n)=PO(n, 1)$, with $k\not = \frac{n}{2}$. On the other hand, for sequences of closed Riemannian manifolds that converge in the Benjamini-Schramm sense, it is not necessarily the case that the injectivity radius goes to infinity (even if the pointed injectivity radius goes to infinity almost everywhere). Thus, we modify the dimension estimate for $\mathcal{H}^k_g(M)$ used in Section 5 of \cite{DS17}, in order to obtain vanishing results in this broader context.

 Let $K(\cdot,\cdot)$ denote the Schwarz kernel for the $L^2$ orthogonal projection onto $\mathcal{H}^{k}_{g}(M)$. Thus for  $x, y \in M$,  
\[
K(x, y) \in Hom( \Omega^{k}T^{*}_{y}M , \Omega^{k}T^{*}_{x}M).
\]
Given an $L^2$-orthonormal basis $\{\alpha_j\}_{j=1}^l$ for $\mathcal{H}^{k}_{g}(M)$, we have 
\begin{align}\label{KMap}
K(x, y) =\sum^{l}_{i=1}\alpha_{j}(x)\langle\cdot, \alpha_{j}(y)\rangle.
\end{align}
  Next, we derive a pointwise estimate on the trace of $K(x, x)$.

\begin{lemma}\label{keyestimate}
	Given  $ K(\cdot, \cdot)$ as above, we have for any $x\in M$
	\[
	0\leq TrK(x, x)\leq \binom{n}{k}\sup_{\alpha\in\mathcal{H}^{k}_{g}(M):||\alpha||^{2}_{L^{2}}=1}|\alpha(x)|^{2}.
	\] 
	\end{lemma}
\begin{proof}
	
	Fix a point $x\in M$, and let  $\{e_{i}\}^{\binom{n}{k}}_{i=1}$ be a local orthonormal frame for $\Omega^{k}T^{*}M$ in a neighborhood of $x$. 
	Then
	\[
	Tr K(x, x)=\sum^{\binom{n}{k}}_{i=1}\langle K(x, x)(e_{i}), e_{i}\rangle=\sum^{l}_{i=1}|\alpha_{i}|^{2}_{x}\geq 0.
	\]
	Next, given a point $p\in M$, there exists a unit eigenvector $z$ of $K(p, p)$ with maximal eigenvalue say $\lambda$. Thus, by construction
	\[
	\langle K(p, p)z, z\rangle =\lambda |z|^{2}_{p}=\lambda,
	\]
	with
	\[
	\langle K(p, p)z, z\rangle=\sum^{l}_{i=1}\langle z, \alpha_{i}(p)\rangle\langle z, \alpha_{i}(p)\rangle.
	\]
	Thus 
	\begin{align}\notag
	&\int_{M}\langle K(x, p)z,  K(x, p)z\rangle d\mu_{g}\\ \notag
	 =&\int_{M}\big\langle \sum^{l}_{i=1}\alpha_{i}(x)\langle z, \alpha_{i}(p)\rangle, \sum^{l}_{j=1}\alpha_{j}(x)\langle z, \alpha_{j}(p)\rangle \big\rangle d\mu_{g}\\ \notag
	 =&\sum^{l}_{i, j=1}\langle z, \alpha_{i}(p)\rangle \langle z, \alpha_{j}(p)\rangle\int_{M}\langle \alpha_{i}(x), \alpha_{j}(x)\rangle d\mu_{g}\\ \notag
	 =&\sum^{l}_{i, j=1}\langle z, \alpha_{i}(p)\rangle\langle z, \alpha_{j}(p)\rangle\delta_{ij}=\lambda.
	\end{align}
	Now, set 
	\[
	\alpha(x):=\frac{K(x, p)z}{\sqrt{\lambda}}\in \mathcal{H}^{k}_{g}(M),
	\]
	with
	\[
	\|\alpha\|_{L^2}=1.
	\]
	In sum, we have found an  $\alpha\in\mathcal{H}^{k}_{g}(M)$ such that
	\[
	||\alpha||^{2}_{L^{2}}=1, \quad |\alpha(p)|^{2}=\lambda.
	\]
	As $\lambda$ was the largest eigenvalue of $K(p,p)$, we have the estimate
	\[
	TrK(p, p)\leq \binom{n}{k}\lambda\leq \binom{n}{k}\sup_{||\alpha||^{2}_{L^{2}}=1}|\alpha(p)|^{2}.
	\]
	Since $p$ is an arbitrary point in $M$, the proof is complete.
\end{proof}

The following lemma is the usual elliptic regularity for harmonic forms in bounded geometry.
 One proof is a standard application of Moser iteration. See for example \cite[Proposition 2.2]{LS}, where the theorem is proved for hyperbolic manifolds and Proposition \ref{bndrymoser}, where it is proved for manifolds with boundary.  
\begin{lemma}\label{Moser}
	Let $(M^{n}, g)$ be a closed Riemannian manifold with 
	\[
	-a\leq\sec_{g}\leq 1,
	\]
	and let
	\[
	\inj_g(M):=\min_{p\in M}\inj_{g}(p)>0
	\]
	be the global injectivity radius. Given a harmonic $k$-form $\alpha\in\mathcal{H}^{k}_{g}(M)$, for any $p\in M$ and $L<\min(\inj_{g}(M), 1)$ there exists a strictly positive constant $d(n, a,k, L):= d(n,a, k)(1+\frac{1}{L})^n$ such that
	\[
	||\alpha||^{2}_{L^{\infty}(B_{\frac{L}{2}}(p))}\leq d(n,a, k, L)||\alpha||^{2}_{L^{2}(B_{L}(p))}.
	\]
\end{lemma}

	

Combining Lemma \ref{keyestimate} with Lemma \ref{Moser}, we get the key estimate of this section.

\begin{lemma}\label{estimate}
	Given $(M^n, g)$,  and $ K(\cdot,\cdot)$ as above, there exists a constant $d_0=d_0(n,a, k, \inj_g(M))>0$ such that
	\[
	0\leq TrK(x, x)\leq d_0(n, a,k, \inj_g(M)),
	\]
	 for any $x\in M$.
\end{lemma}
\begin{proof}
	By Lemma \ref{keyestimate}, we have
	\[
	0\leq TrK(x, x)\leq \binom{n}{k}\sup_{\alpha\in\mathcal{H}^k_g(M):||\alpha||^{2}_{L^{2}}=1}|\alpha(x)|^{2}.
	\]
	  Now apply  Lemma \ref{Moser} to obtain the desired estimate.
\end{proof}

\section{Negative Ricci Curvature}\label{negativericci}

In this section, we study manifolds with negative Ricci curvature. 

\begin{define}
	Let $(M, g)$ be a complete Riemannian manifold. Given any $R>0$, we define the $R$-thin part of $(M, g)$ as
	\[
	M_{<R}:=\{x\in M\,\, | \,\, \inj_g(x)< R\},
	\]
	where $\inj_{g}(x)$ is the injectivity radius of $(M, g)$ at the point $x$. We define the $R$-thick part, denoted by $M_{\geq R}$, as the complement of the $R$-thin part.
\end{define}

The proof of \cite[Theorem 66]{DS17}, implies the following theorem. 

\begin{theorem}\label{PriceRicci}
	Let $(M^{n}, g)$ be a compact manifold  with $-1\leq \sec_{g}\leq 1$. 
	Given $k\in \IN$, assume there exists $\delta>4k^2$ such that
	\[
	-Ric \geq \delta g.
	\]
	Let $\rho$ be large enough so that
	\[
	\frac{\sqrt{\delta}}{2}\coth(\sqrt{\delta}\rho)-k\coth{(\rho)}\geq \epsilon >0.
	\]
	There exists $c(n, k, \delta, \epsilon)>0$ so that for any  $\alpha\in \mathcal{H}^k_g(M)$ and    $p\in M$ with $\inj_g(p)>\rho+2$, we have    
	\begin{align}
	\int_{B_\rho(p)}|\alpha|^2dv\leq   c(n, k, \delta,\epsilon)e^{-(\sqrt{\delta}-2k)(\inj_g(p)-\rho-2)}\|\alpha\|^2_{L^2(M, g)}.
	\end{align}
\end{theorem}

A corollary of this estimate is the following result for sequences of Riemannian manifolds which Benjamini-Schramm converge to their universal cover.

\begin{theorem}\label{RicciBS}
	Let $(X^n, g)$ be a simply connected manifold without conjugate points with $-1\leq\sec_g\leq 1$. Assume there exists $\delta>0$ such that
	\[
	-Ric_g\geq \delta g.
	\]
	Let $(M_l, g_l)$ be a weakly uniformly discrete sequence of closed manifolds converging to $(X, g)$. Then for any $k\in \IN$ such that $\delta>4k^2$, we have
	\[
	\lim_{l\to\infty}\frac{b_k(M_l)}{Vol_{g_l}(M_l)}=0.
	\]
\end{theorem}

\begin{proof}
	 Observe that for any $(k,R)\in\IN\times (0,\infty)$ such that $\delta>4k^2$ and $R>max\{\rho+2, R_{0}\}$, with $\rho$ as defined in Theorem \ref{PriceRicci}, we have the estimate:
	\begin{align} 
	\frac{b_k(M_l)}{Vol_{g_l}(M_l)}&=\frac{\int_{(M_l)_{<R}}TrK(x, x)d\mu_{g_l}}{Vol_{g_l}(M_l)}+\frac{\int_{(M_l)_{\geq R}}TrK(x, x)d\mu_{g_l}}{Vol_{g_l}(M_l)}\\ \notag
	&\leq\binom{n}{k}d(n,a,k)\Big(1+ \frac{1}{\inj_{g_l}(M_l) }\Big)^n\rho(M_l,R)+c(n, k, \delta)e^{-(\sqrt{\delta}-2k)(R-\rho-2)} .
	\end{align}
Choose $R= R_l$ for some sequence $\{R_j\}_j$ given by the definition of weakly uniformly discrete (Definition \ref{weakdisc}), and the result follows. 
\end{proof}

\begin{remark}
	Theorem 122 in \cite{DS17} implies that, under the curvature assumptions of Theorem \ref{RicciBS}, we have
	\[
	b^{(2)}_{k}(M_{l}):=\dim_{\Gamma_l}(\mathcal{H}^k_{2}(X))=0,
	\]
	for any $k\in\IN$ such that $\delta>4k^2$. Thus, Theorem \ref{RicciBS} can alternatively be rephrased by saying that the normalized $k$-Betti number converge along the sequence to the corresponding $k$-th $L^2$-Betti number.
\end{remark}




\section{Pinched Negative Sectional Curvature}

In this section, we study sequences of Riemannian manifolds with negative and pinched sectional curvature which Benjamini-Schramm converge. The starting point is as usual the Price inequality for harmonic forms on manifolds with negative sectional curvature established in \cite{DS17}.

\subsection{Uniformly Discrete Sequences}

We start with sequences of uniformly discrete, negatively curved and pinched manifolds which BS-converge. The key technical point is a Price inequality for harmonic $k$-forms. For convenience, we assemble in a single statement three results stated distinctly in \cite{DS17}.

\begin{theorem}[Theorems 87 \& 96  and Corollary 108 in \cite{DS17}]\label{pinchedprice}
	Let $(M^n, g)$ be a compact manifold of dimension $n\geq 3$. Assume the sectional curvature is pinched:
	\[
	-a^2\leq \sec_g\leq -1
	\]
	with $a\geq1$. Let $k$ be a non-negative integer such that
	\[
	a_{n, k}:=(n-1)-2ka>0.
	\]
	 There exists a constant $c(n, k)>0$, so that for any $\alpha\in\mathcal{H}^k_g(M)$ and for any geodesic ball $B_R(p)\subset M$, with $ 1+\frac{\ln{(2)}}{a_{n, k}}<R< \inj_g(p)$, 
	\[
	\int_{B_{1}(p)}|\alpha|^2dv\leq c(n, k)e^{-a_{n, k}R}\|\alpha\|^2_{L^2(M, g)}.
	\]
	Finally, if $k$ is a non-negative integer such that
	\[
	a_{n, k}:=(n-1)-2ka=0,
	\]
	then there exists a constant $d(n, k)>0$, so that for any geodesic ball $B_R(p)\subset M$, with $1<R<\inj_g(p)$, 
	\[
	\int_{B_{1}(p)}|\alpha|^2dv\leq d(n, k)(R-1)^{-1}\|\alpha\|^2_{L^2(M, g)}.
	\]
\end{theorem}
\begin{remark}The proof of this theorem requires only that $\alpha$ be closed and coclosed in $B_R(p)$ and that the curvature pinching holds within this ball. Hence the result extends to manifolds with boundary (if $d(p,\p M)\geq R$), with any boundary condition, and to noncompact manifolds. 
\end{remark}
As in Section \ref{negativericci}, a Price inequality has an immediate consequence for weakly uniformly discrete sequences of Riemannian manifolds which Benjamini-Schramm converge.

\begin{corollary}\label{sectionaluniform}
Let $(X^n, g)$ be a simply connected manifold of dimension $n\geq 3$ with
\[
-a^2\leq\sec_g\leq-1,
\]	
and $a\geq 1$. Let $(M_l, g_l)$ be a weakly uniformly discrete sequence of closed manifolds converging to $(X, g)$. For any $k\in \IN$ such that
\[
a_{n, k}=(n-1)-2ka\geq 0,
\]
we have 
\[
\lim_{l\to\infty}\frac{b_k(M_l)}{Vol_{g_l}(M_l)}= 0.
\]
\end{corollary}

\begin{remark}
	Proposition 126 in \cite{DS17} implies that, under the curvature assumptions of Corollary \ref{sectionaluniform}, we have
	\[
	b^{(2)}_{k}(M_{l}):=\dim_{\Gamma_l}(\mathcal{H}^k_{2}(X))=0,
	\]
	for any $k\in\IN$ such that $a_{n, k}\geq 0$ (\emph{cf.} also Proposition 4.1 in \cite{Donnelly} when $a_{n, k}>0$, and \cite{Dodziuk} when $a=1$). Thus, Corollary \ref{sectionaluniform} can alternatively be rephrased by saying that the normalized $k$-Betti number converge along the sequence to the corresponding $k$-th $L^2$-Betti number.
\end{remark}

\subsection{Thick-Thin Decomposition Revisited}\label{newTT}

In this section, we construct an effective Thick--Thin decomposition for closed manifolds with negative sectional curvature. This Thick-Thin decomposition builds upon a construction of Buser, Colbois, and Dodziuk (\emph{cf.} \cite{BCD93}) which we refine for our purposes.  These additions are needed for our Moser iteration argument on manifolds with boundary (\emph{cf.} Section \ref{moserboundary}). 

Let $(M, g)$ be a compact of dimension $n\geq 3$, with sectional curvatures satisfying
\[
-a^2\leq \sec_g\leq -1,
\]
for some constant $a\geq 1$. A first geometric consequence of the so-called Gromov-Margulis' Lemma (\emph{cf.} Section 8 in the book \cite{BGS85}) is that there is a positive constant
\[
\mu=a^{-1}c_{n},
\]
$c_n>0$ depending on the dimension only, so that if the set
\[
M_{\mu}:=\{x\in M\,\, | \,\, \inj(x)<\mu \}
\]
is not empty, it is then the union of a finite number of disjoint tubes $\{T_{\gamma_i}\}$ around short closed geodesics $\{\gamma_i\}$. For convenience, we will require $\mu< 1$. For every tube $T_{\gamma}$, the core geodesic $\gamma$ has length $l(\gamma)<2\mu$. For every point $p\in \gamma$, and every tangent vector $v\in T_{p}M$ perpendicular to $\gamma'(0)$, let $\delta_{p, v}(t)$ denote the unit speed geodesic ray emanating from $p$ in the direction of $v$. We call these rays {\em radial arcs} and their tangent vector fields the radial vector field $\mathcal{R}$. 
\begin{lemma}\label{bluelemma}
	Let $\gamma$ be a geodesic in $M$ satisfying \eqref{newconst}.  Then
	  \begin{align}\label{nabr}div( \mathcal{R})(x)\leq (n-1)a\coth(ad(x,\gamma)).\end{align} 
\end{lemma}
\begin{proof}
The proof is a Riccati comparison argument, which we include for the convenience of the reader. (See \cite[Section 2]{Ch01}.) 
Let $\rho$ denote distance to $\gamma$. Then $|d\rho |=1$, and the Bochner formula gives 
\begin{align}\label{bochm}
0 = \Delta\frac{1}{2}|d\rho|^2  = -|\nabla d\rho|^2-Ric(d\rho,d\rho) +\langle d\Delta\rho,d\rho\rangle.
\end{align}
Set $m:= -\Delta\rho = div\mathcal{R}$. Since $\nabla_{\mathcal{R}}\mathcal{R} = 0,$ $Hess(\rho)$ has rank $\leq n-1$. Hence Cauchy-Schwarz implies 
$m^2\leq (n-1)|\nabla d\rho|^2.$ Applying these inequalities to \eqref{bochm} yields the following equation  for the value of $m=m(t)$ along a geodesic ray $\delta_{p, v}(t)$.
\begin{align}\label{bochm2}
 \frac{\p m}{\p t}+\frac{m^2}{n-1}\leq - Ric(d\rho,d\rho)\leq (n-1)a^2.
\end{align}
Set $y:= (n-1)a\coth( at).$ Then $\dot y + \frac{y^2}{n-1}= (n-1)a^2.$ Thus we have 
\begin{align}\label{bochm3}
 \frac{\p (m-y)}{\p t}+\frac{(m+y)}{n-1}(m-y)\leq   0.
\end{align}
$\lim_{t\to 0}(m-y) = -\infty.$ Hence the Riccati comparison principle \cite[Subsection 6.4.1]{Petersen} implies 
$$m(t)\leq y(t),$$ 
for all $t>0$. 

\end{proof}
In every interval $[0, t_0]$ such that $i(\delta_{p, v}(t))\leq \mu$ for $t\in [0, t_0]$,  the function $t\rightarrow i(\delta_{p, v}(t))$ is strictly monotonic increasing. Thus, there exists $R_{p, v}>0$ depending on the initial condition of the geodesic ray such that $i(\delta(R_{p, v}))=\mu$ and $i(\delta_{p, v}(t))<\mu$ for any $t\in[0, R_{p, v})$. The arc $\delta([0, R_{p, v}])$ is called the maximal arc. Also, different radial arcs are disjoint except possibly for their initial points. Thus, their union $T_{\gamma}$ is homeomorphic to $\gamma\times B^{n-1}$ where $B^{n-1}$ is a closed ball inside $\IR^{n-1}$, and this explains why they are called tubes. On the other hand, different maximal radial arcs in $T_{\gamma}$ may have very different lengths and the boundary of $T_{\gamma}$ is not smooth in general. Thus, these object are not great if you want to do calculus on them. We therefore employ a controlled Thick-Thin decomposition due to Buser, Colbois and Dodziuk \cite{BCD93} which we now briefly describe. First, we state the following lemma, observed in \cite{BCD93}.

\begin{lemma}\label{BCD1}
	There exist constants $c_1$, $c_2$ depending only on the dimension $n$, such that if
	\begin{align}\label{newconst}
	l(\gamma)\leq c_1\exp(-c_2 a)\mu^n a^{n-1},
	\end{align}
	then $d(x, \gamma)\geq 10$ for every $x\in T_{\gamma}$ with $\inj(x)=\frac{\mu}{2}$.
\end{lemma}
\begin{proof}
	This is a consequence of Lemma 2.4 in \cite{BCD93}.
\end{proof}

From now on, we will regard a geodesic small if and only if its length satisfies the bound \eqref{newconst} of  Lemma \ref{BCD1}. This means we disregard possibly many small geodesics in the usual Thick-Thin decomposition of $M$. Thus, we only look at small geodesics which posses fat Margulis tubes around them. This fact plays a role in the constructions that follow.\\

Next, given a geodesic $\gamma$ satisfying \eqref{newconst} and $\lambda\in (0,1)$, we define the following  tube around it:
\begin{align}\label{newtube}
U_{\gamma}^\lambda:=\{x\in T_{\gamma}\,\, | \,\, \inj(x)\leq \lambda\mu \}.
\end{align}
Again, there is no a priori reason to believe that the boundary $\partial U_{\gamma}^\lambda$ is well behaved from a geometric point of view. Thus, we appeal to the following theorem of Buser, Colbois and Dodziuk which ensures the existence of a small deformation of $U_{\gamma}^\frac{1}{2}$ with many nice geometric properties.

\begin{theorem}[Theorem 2.14 in \cite{BCD93}]\label{mainBCD}
	Let $\gamma$ be a geodesic in $M$ satisfying \eqref{newconst}. There exists a smooth hypersurface $H_{\gamma}$ contained in $T_{\gamma}\setminus\gamma$ with the following properties:
	\begin{itemize}
		\item The angle $\theta$ between the radial vector field $\mathcal{R}$ and the exterior normal of $H_{\gamma}$ is less that $\pi/2-\alpha$ for a constant $\alpha=\alpha(a, n)\in (0, \pi/2)$. 
		\item The sectional curvatures of $H_{\gamma}$ with respect to the induced metric are bounded in absolute value by a constant depending only on $a$ and $n$.
		\item $H_{\gamma}$ is homeomorphic to $\partial U_{\gamma}^{\frac{1}{2}}$ by pushing along radial arcs. The distance between $x\in H_{\gamma}$ and its image $\bar{x}\in\partial U_{\gamma}^{\frac{1}{2}}$ satisfies $d(x, \bar{x})\leq\mu/50$.
	\end{itemize}
\end{theorem}  

Next, we explicitly observe the following consequence of Theorem \ref{mainBCD}. This corollary plays a role in the elliptic estimates presented in Section \ref{moserboundary}.

\begin{corollary}\label{Lconst}
$H_{\gamma}$ is locally the graph of a Lipschitz function with Lipschitz constant $\Lambda_H$ dependent only on $a$ and $n.$
\end{corollary}

\begin{proof}This follows readily from the estimates in \cite[p.12]{BCD93} required to prove Theorem \ref{mainBCD}, in particular from the multiplicative bounds on the gradient of the defining function of $H_\gamma$. 
\end{proof}

We can now define the tubes in our refined Thick-Thin decomposition. Given a geodesic $\gamma$ satisfying \eqref{newconst}, we consider the tube $V_\gamma$ around it defined by:
\begin{align}\label{finaltube}
H_{\gamma}=\partial V_{\gamma},\text{ and }\gamma\in V_\gamma.
\end{align}
In particular, these new tubes always have \emph{smooth} boundaries. We now derive a few lemmas concerning the tubes $V_{\gamma}$ that are not directly found in \cite{BCD93}; so, we provide all details of the proof. Let 
 
\begin{lemma}\label{lbound}
	Let $\gamma$ be a geodesic in $M$ satisfying \eqref{newconst}. For any point $x\in H_{\gamma}$, we have
	\[
	\frac{26}{50}\mu\geq \inj_g(x)\geq \frac{24}{50}\mu.
	\] 
In particular, $  U_\gamma^{\frac{24}{50}}\subset V_\gamma\subset U_\gamma^{\frac{26}{50}}.$ 
\end{lemma}
\begin{proof}
	Given $x\in H_{\gamma}$, denote by $\bar{x}$ the point of intersection with $\partial U_{\gamma}$ of the radial arc from $\gamma$ to $x$. We have the standard estimate:
	\begin{align}\label{standest}
	d(x, \bar{x})\geq |\inj_g(x)-\inj_g(\bar{x})|.
	\end{align}
	By definition of $U_{\gamma}$, we have $\inj_g(\bar{x})=\frac{\mu}{2}$. Thus, by Theorem \ref{mainBCD} we have:
	\[
	\frac{\mu}{50}+\frac{\mu}{2}=\frac{26}{50}\mu\geq \inj_g(x)\geq \frac{\mu}{2}-\frac{\mu}{50}=\frac{24}{50}\mu.
	\]
\end{proof}

Next, we show that these tubes are uniformly separated.

\begin{lemma}\label{lapart}
	If $\gamma\neq \zeta$ are two distinct closed geodesics in $M$ satisfying \eqref{newconst}, then
	\[
	d(V_{\gamma}, V_{\zeta})>\frac{48}{50}\mu.
	\]
\end{lemma}
\begin{proof}
	Let $\beta$ be a unit speed geodesic realizing the distance between the compact sets $V_{\gamma}$ and $V_{\zeta}$. There exist $t_1, t_2\in (0, d(V_{\gamma}, V_{\zeta}))$ such that $\beta(t_1)\in \partial T_{\gamma}$ and $\beta(t_2)\in \partial T_{\zeta}$. By Lemma \ref{lbound} we have 
	\[
	t_1\geq \mu-\inj_g(\beta(t_1))\geq \mu-\frac{26}{50}\mu, \quad (d(V_{\gamma}, V_{\zeta})-t_{2})\geq \mu-\frac{26}{50}\mu,
	\]
	so that
	\[
	 d(V_{\gamma}, V_{\zeta})\geq \frac{24}{50}\mu+\frac{24}{50}\mu+d(T_{\gamma}, T_{\zeta})>\frac{48}{50}\mu.
	\]
	Note that by the usual Thick-Thin decomposition, the Margulis tubes $T_{\gamma}$ and $T_{\zeta}$ are disjoint.
\end{proof}

We also need to know that each $V_{\gamma}$ contains a ``quantum'' of volume. This is essential in studying sequences that BS-converge (\emph{cf.} Lemma \ref{numbertubes}). In \cite{BCD93}, the authors show this is the case for the tubes $T_{\gamma}$. We show that their argument can be extended to the tubes $V_{\gamma}\subset T_{\gamma}$.

\begin{lemma}\label{quanta}
	If $\gamma$ is a closed geodesic satisfying \eqref{newconst}, then
\begin{align}Vol(  U_{\gamma}^\lambda)>c_{n}\Big(\frac{\lambda}{2}\mu\Big)^{n},
\end{align}
and therefore 
	\[
	Vol(V_{\gamma})>c_{n}\Big(\frac{6}{25}\mu\Big)^{n}.
	\]
\end{lemma}
\begin{proof}
	Let $x\in  \partial U_{\gamma}^{\frac{\lambda}{2}}$, and let $y$ be such that $inj_g(y)=\lambda\mu/2$. 
	We  claim that $B_{\frac{1}{2}\lambda \mu}(y)$, the open geodesic ball of radius $\frac{1}{2}\lambda \mu$ centered at $y$,  is entirely contained in $U_{\gamma}^\lambda$. In fact, for any $z\in B_{\frac{1}{2}\lambda \mu}(y)$ we have
	\[
	\frac{1}{2}\lambda \mu>d(z, y)\geq \Big|\inj_g(z)-\frac{1}{2}\lambda\mu\Big|,
	\]
	which forces $\inj_g(z)<\lambda\mu$. Thus, we conclude that $B_{\frac{1}{2}\lambda \mu}(y)\subset   U_{\gamma}^\lambda$.
	By volume comparison with Euclidean space we have:
	\[
	 Vol(  U_\gamma^\lambda) > Vol\Big(B_{\frac{1}{2}\lambda \mu}(y)\Big)\geq c_{n}\Big(\frac{1}{2}\lambda\mu\Big)^{n}.
	\]
Since $V_\gamma\supset U_\gamma^{\frac{24}{50}},$ we have $Vol(V_\gamma)\geq c_{n}\Big(\frac{6}{25}\lambda\mu\Big)^{n}.$
\end{proof}
 
We continue deriving effective estimates for the sizes of tubes in the Buser-Colbois-Dodziuk thick-thin decomposition.\\

Let $\delta_{p, v}(t)$ be a unit speed radial arc, and let  $t=R_\lambda$ be the first time  the radial arc intersects the boundary of $  U_{\gamma}^\lambda$.
We have the estimate
\begin{align}\label{Restimate}
c_{n}\Big(\frac{\lambda\mu}{2}\Big)^n \leq l(\gamma)a^{-(n-1)}\sinh^{n-1}(aR_\lambda).
\end{align}
In order to prove \eqref{Restimate}, we argue as follows. Let $y\in \delta_{p, v}([0, R_{\lambda}])$ be a point such that $\inj_g(y)=\frac{1}{2}\lambda\mu$. Then the inclusion $B_{\frac{\lambda\mu}{2}}(y)\subset   U_\gamma^\lambda$ follows from the proof of Lemma \ref{quanta}. 
Next, we claim that 
\[
B_{\frac{\lambda\mu}{2}}(y)\subset \{z\in   U_{\gamma}^\lambda \,\, | \,\, d(z, \gamma)\leq R_\lambda\}.
\]
This follows from the triangle inequality since for any $z\in  B_{\frac{\lambda\mu}{2}}(y)$
\[
d(z, \gamma)\leq d(\gamma, y)+d(y, z)\leq R_\lambda-\frac{\lambda\mu}{2}+\frac{\lambda\mu}{2}=R_\lambda.
\]
Thus, by volume comparison with a space of constant sectional curvature  $-a^2$, we obtain the claimed inequality \eqref{Restimate}.
We therefore conclude there exists $k=k(a, n,\lambda)>0$ such that 
\begin{align}\label{rlower}
R_\lambda\geq \frac{\ln(\frac{1}{l(\gamma)})}{(n-1)a}+ k(a, n,\lambda).
\end{align}
Thus, if we denote by $R_{min,\lambda}(\gamma)$ the length of the shortest radial arc reaching $\partial   U_{\gamma}^\lambda$ we have that 
\[
\lim_{l(\gamma)\to 0}R_{min,\lambda}(\gamma)\to\infty.
\]

\subsection{Pointwise Bounds for Harmonic Forms on Manifolds with Boundary}\label{moserboundary}
  
Let $M_T:= M\setminus \cup_{\gamma} V_\gamma^\circ$, where the $V_\gamma$ are the modified tubes defined in \eqref{finaltube}. We will need  elliptic estimates for harmonic forms in $M_T$ satisfying Neumann or Dirichlet boundary conditions on $\p M_T$.  In order to keep track of both the dependence of the estimates on the geometry of the Margulis tube and on the local injectivity radius, in this section we provide a proof of these estimates. 

Let  $\|f\|_{L_{1}^2(B_{R}(p),g)}$ and $\|f\|_{L_{1}^2(B_{R}(p),\text{euclidean})}$ denote the Sobolev norm 
\[
\|f\|^2_{L^2(B_{R}(p))}+\|df\|^2_{L^2(B_{R}(p))}
\] 
computed with respect to $g$ and with respect to the Euclidean metric induced by the exponential map, respectively. 
\begin{proposition}\label{coro1}
Let $(M^n,g)$ be compact with $-a^2\leq sec_g\leq -1$. There is a constant $S(a,n)$ depending only on $a$ and $n$ so that 
	for all geodesic balls  
	$B_R(p)\subset M$, with $R\leq 1$,  and all   $ f\in C_c^\infty(B_{R}(p)\cap M_T)$,  one has 
	\begin{align}\label{sob3}S(a,n) \|  f\|^2_{L^2_1(B_R(p)\cap M_T)}\geq \| f\|_{L^{\frac{2n}{n-2}}(B_R(p)\cap M_T)}^2.
	\end{align}
\end{proposition}
\begin{proof}For 
$R<\min\{\inj_g(p),1\}$, there are constants $c_{a,j}$ and $C_{a,j}$, $j=0,1$  depending only on $a$ and $n$ so that for any domain $A\subset B_R(p)$,
\begin{align}c_{a,0}  \|f\|_{L^{\frac{2n}{n-2}}(A,\text{euclidean})}\leq \|f\|_{L^{\frac{2n}{n-2}}(A,g)}\leq C_{a, 0}\|f\|_{L^2(A,\text{euclidean})},\end{align}
and
 \begin{align}c_{a,1}  \|f\|_{L_{1}^2(A,\text{euclidean})}\leq \|f\|_{L_{1}^2(A,g)}\leq C_{a,1}\|f\|_{L_{1}^2(A,\text{euclidean})}.\end{align}
 By \cite[Theorem 5]{Stein} and its proof (see also \cite{Jones}), there exists a bounded extension map 
$$E_{T,R,p}: W_0^{1,2}(B_{R}(p)\cap M_T,\text{euclidean})\to L_{1}^2(\IR^n,\text{euclidean}),$$ satisfying 
$E_{T,R,p}f -f = 0$ on $B_{R}(p)\cap M_T$ and 
$$\|E_{T,R,p}f\|_{L_{1}^2(\IR^n,\text{euclidean})}\leq B(\Lambda_H,n)\|f\|_{L_{1}^2(B_{R}(p),\text{euclidean})},$$ with bound $B(\Lambda_H,n)$ depending only on dimension and  the Lipschitz constant $\Lambda_H$ for a defining function for $\p M_T$ (\emph{cf.} Corollary \ref{Lconst}). Hence we have 
\begin{align}
\| f\|_{L^{\frac{2n}{n-2}}(B_R(p)\cap M_T,g)}^2&\leq C_{a,0}^2\| f\|_{L^{\frac{2n}{n-2}}(B_R(p)\cap M_T,\text{euclidean})}^2
\nonumber\\
&\leq   C_{a,0}^2\frac{(n-1)^2}{(n-2)^2}\| d(E_{T,R,p}f)\|_{L^{2}(\IR^n,\text{euclidean})}^2
\nonumber\\
&\leq   C_{a,0}^2\frac{(n-1)^2}{(n-2)^2}B(\Lambda_H,n)^2\|   f \|_{L_{1}^2(B_{R}(p)\cap M_T,\text{euclidean})}^2
\nonumber\\
&\leq   \frac{C_{a,0}^2}{c_{a,1}^2}\frac{(n-1)^2}{(n-2)^2}B(\Lambda_H,n)^2\|   f \|_{L_{1}^2(B_{R}(p)\cap M_T,g)}^2
\end{align}
Set $$S(a,n) := \frac{C_{a,0}^2}{c_{a,1}^2}\frac{(n-1)^2}{(n-2)^2}B(\Lambda_H,n)^2$$ 
to obtain the desired result. 
\end{proof}

Given Proposition \ref{coro1}, we can now prove the main estimate of this section.

\begin{proposition}\label{bndrymoser}
	 	Let $h$ be a strongly harmonic  form on $M_T$ satisfying Dirichlet or Neumann boundary conditions on $\p M_T$.
	Then there exist $c_n\in (0,\infty)$ independent of $M$ and $C_G>0$ depending on the second fundamental form of $H_\gamma$, the Margulis constant of $M$, and the Riemann curvature tensor of $M$ such that for any $p\in M_T$, and $L\leq 1$, 
	\begin{align}
	\| h\|^2_{L^{\infty}(B_{L}(p))}\leq c_{n} S(a,n)^{\frac{n}{2}}\Big(\frac{1}{L^2}+C_G\Big)^{\frac{n}{2}}\| h\|^2_{L^{2}(B_{2L}(p))}.
	\end{align}
In particular, choosing $L= \frac{\mu}{4}$, we have 
\begin{align}\label{runtiprun}
	| h|^2(p)\leq r(a,n)\| h\|^2_{L^{2}(B_{\frac{\mu}{2}}(p))}.
	\end{align}
where $r(a,n) := c_{n} S(a,n)^{\frac{n}{2}}\big(\frac{16}{\mu^2}+C_G\big)^{\frac{n}{2}}$.
	\end{proposition}

\begin{proof} 
By Lemma \ref{lapart}, the connected components of $\partial M_{T}$ are uniformly apart. Without loss of generality, we assume $h$ satisfies Dirichlet boundary conditions: the pullback to the boundary of $h$ and $d^*h$  vanishes. 
	Let $\psi$ be a smooth function supported in $M_T$ (but not necessarily compactly supported in $M_T^0$).  
	Then we have 
\begin{align}\label{nobc}0 &= \int_{M_T} \langle \Delta h, \psi^2h \rangle dv  \nonumber\\
&=\int_{M_T} (\langle \nabla h, \nabla(\psi^2h) \rangle + \langle R^{riem} h, \psi^2h \rangle)dv-\int_{\p M_T} \langle \nabla_\nu h, \psi^2h \rangle d\sigma,
	\end{align}
where $\nu$ is an outward pointing unit normal and $R^{riem}$ denotes the curvature operator given in a local orthonormal frame $\{e_i\}_i$ and dual coframe $\{\omega^i\}_i$ by 
$R^{riem} = -e(\omega^i)e^*(\omega^j)R^{riem}(e_i,e_j),$ with $e(\omega^p)$ denoting exterior multiplication on the left by $\omega^p$, $e^*(\omega^p)$ its adjoint, and $R^{riem}(\cdot,\cdot)$ the Riemannian curvature 2 form. Since $h$ is strongly harmonic, we have 
\[
dh=\sum_{j\geq 1}e(\omega^i)\nabla_{e_i}h=0, \quad d^{*}h=-\sum_{j\geq 1}e^{*}(\omega^i)\nabla_{e_i}h=0,
\]
so that if $e_1=\nu$ on $\p M_T$, we obtain the identities
\begin{align}\nabla_\nu h = -\sum_{j>1}\langle (e^*(\omega^1) e(\omega^j)-  e^*(\omega^j)e(\omega^1))\nabla_{e_j}h.
\end{align}
Since $h$ satisfies Dirichlet boundary conditions, we have $e(\omega^1)h=0$ on $\partial M_{T}$. Thus, we have
\begin{align}\label{secf}-\langle \nabla_\nu h,\psi^2h\rangle =  -\sum_{j>1}\langle ( e^*(\omega^j)[e(\omega^1),\nabla_j]h,\psi^2h\rangle
=  \langle  II h,\psi^2h\rangle,
\end{align}
where $II$ denotes the second fundamental form operator 
$$II= II_{jk}e^*(\omega^j) e( \omega^k).$$
From \eqref{nobc} and \eqref{secf}, we obtain a Bochner formula for harmonic forms satisfying Dirichlet boundary conditions: 
	
	\begin{align}\label{bochfin}
	0&= \int_{M_T}( \langle \nabla  h, \nabla (\psi^2h)\rangle  + \langle  R^{riem} h ,  \psi^2h \rangle) dv +\int_{\p M_T} \langle   II h,   \psi^2h \rangle d\sigma.
	\end{align}
	This equality now allows us to proceed with the usual proof of Moser iteration with one additional modification required for controlling the boundary term. We repeat the standard argument and add in the contribution from the boundary. We will follow the treatment in \cite{cls1}.  
	
	Given $p\in M_T$,  let $\eta:\IR\to [0,1]$ be a smooth function identically $1$ on $(-\infty,L ]$ and supported on $(-\infty,2L]$, with $|d\eta|\leq \frac{2}{L}$. Set $\eta_k(x) = \eta(2^k(d(x,p)-L))$. Observe that the function $\eta_k(t)$ is equal to one on  $(-\infty,L(1+2^{-k})]$, and it is supported on $(-\infty,L(1+2^{1-k})]$.   Let $\chi_k$ denote the characteristic function of $B_k:=B(p,L(1+2^{-k}))$. Then 
	\begin{equation}\label{chibnd}
	\chi_k\leq  \eta_k \leq \chi_{k-1}, 
	\end{equation}
	and
	\begin{equation}\label{dchibnd}
	|d\eta_k|\leq \frac{2^{k+1}}{L}\chi_{k-1}.
	\end{equation}
	
	Now we choose $\psi = \eta_k|h|^{p_k-1}$ in \eqref{bochfin}  with $p_k$ to be chosen later to get 
	
	\begin{align}\label{moser0}
	\int_{M_T} |\nabla (\eta_k|h|^{p_k-1} h)|^2 dv &=\int_{M_T} (|d( \eta_k|h|^{p_k-1})|h| |^2   -  \langle  R^{riem} h ,\eta_k^2|h|^{2p_k-2}h \rangle) dv\nonumber\\
	&-\int_{\p M_T} \langle  II h,  \eta_k^2|h|^{2p_k-2}h \rangle d\sigma.
	\end{align}
	Expanding and rearranging terms yields 
	
	\begin{align}\label{moser0.5}
	\int_{M_T} |\nabla (\eta_k|h|^{p_k-1} h)|^2 dv &=\int_{M_T} \Big| \Big(\frac{p_k-1}{p_k}\Big)d(\eta_k|h|^{p_k})+\frac{1}{p_k}d(\eta_k)|h|^{p_k}   \Big|^2  dv\nonumber\\
	&-\int_{M_T} \langle  R^{riem} h ,   \eta_k^2|h|^{2p_k-2}h \rangle dv\nonumber\\
	&-\int_{\p M_T} \langle   II h,  \eta_k^2|h|^{2p_k-2}h \rangle d\sigma.
	\end{align}
	
	Observe that after some manipulation
	\begin{align}\notag
	\int_{M_T}& \Big| \Big(\frac{p_k-1}{p_k}\Big)d(\eta_k|h|^{p_k})+\frac{1}{p_k}d(\eta_k)|h|^{p_k}   \Big|^2  dv\nonumber\\
	&\leq\Big(1-\frac{1}{p_k}\Big)\int_{M_T} |d(\eta_k|h|^{p_k})|^2 dv + \frac{1}{p_{k}}\int_{M_T} |d(\eta_k)|h|^{p_k}|^2 dv.\nonumber
	\end{align}
	
 Now Kato's inequality implies the pointwise bound
 \[
 |\nabla (\eta_k|h|^{p_k-1} h)|\geq |d(\eta_k|h|^{p_k} ) |,
 \]
so that
	
	\begin{align}\label{moser1}
	\frac{1}{p_k} \int_{M_T} |d (\eta_k|h|^{p_k}  )|^2 dv&\leq 
	\frac{1}{p_k}\int_{M_T}|d(\eta_k)|h|^{p_k} |^2  dv+C_R\int_{M_T}   \eta_k^2 |h|^{2p_k} dv\nonumber\\
	&+C_{II}\int_{\p M_T}   \eta_k^2 |h|^{2 p_k}  d\sigma,
	\end{align}
where $C_{II}:= \|II\|_{L^\infty(\p M_T)},$ and $C_R:= \| R^{riem}\|_{L^\infty( M_T)}.$
	 Let $\phi$ be a $C^1$ cutoff function compactly supported in $[0,\frac{24\mu }{50})$ satisfying 
$$\phi(t) = 1, \text{ for }t\leq\frac{12\mu}{50},\text{ 
and  }|d\phi(t)| \leq  \frac{100}{12\mu}.$$ 
Then $\phi(d(x,H))$ is Lipschitz with Lipschitz constant $\frac{25}{3\mu}$, and by construction, it is identically equal to one on $\partial M_{T}$. Now, let $\mathcal{R}$ be the radial vector field as defined in Theorem \ref{mainBCD}. Also,  let $\theta$ be the angle between $\mathcal{R}$ and the unit normal $\nu$ on $\partial M_{T}$. We then have
\[
\langle\nu,\mathcal{R}\rangle=\cos{\theta}\geq \cos{(\pi/2-\alpha)}=\sin{\alpha}, 
\]
where $\alpha\in (0, \pi/2)$ is as in the first statement of Theorem \ref{mainBCD}. Thus, by applying the divergence theorem to $\mathcal{R}$
	\begin{align}\label{bndry}&C_{II}\int_{\p M_T}    \eta_k^2 |h|^{ 2p_k}  d\sigma    \nonumber\\
&\leq \frac{C_{II}}{\sin(\alpha)}\int_{\p M_T}    \eta_k^2 |h|^{2p_k} \langle\nu,\mathcal{R}\rangle  d\sigma \nonumber\\
&= \frac{C_{II}}{\sin(\alpha)}\int_{  M_T}  d(  \eta_k^2 |h|^{2 p_k}  \phi(d(x,H)) i_{\mathcal{R}}dv)\nonumber\\
	&\leq\frac{C_{II}}{\sin(\alpha)}\|d( \eta_k |h|^{ p_k}) \|_{L^2}\| \phi  \eta_k |h|^{ p_k} \|_{L^2(M_T)}\nonumber \\
	&\quad+\frac{C_{II}}{\sin(\alpha)}\Big(\frac{25}{3\mu}+a\coth(10a)\Big)^2\| \eta_k |h|^{ p_k}\|^2_{L^2(M_T)}
	\nonumber\\
	&\leq \frac{1}{2p_k }\|d( \eta_k |h|^{ p_k})\|_{L^2(M_T)}^2+ \frac{C_{II}}{\sin(\alpha)}\Big(\frac{100}{ \mu^2} + \frac{p_kC_{II} }{2\sin (\alpha)}\Big)\|   \eta_k |h|^{ p_k} \|_{L^2(M_T)}^2,
	\end{align}
 where we have used Lemmas \ref{bluelemma} and \ref{BCD1} to bound the covariant derivative of $i_{\mathcal{R}}$ above. Inserting this inequality back into \eqref{moser1} gives 

\begin{align}\label{moser2}
	  \|   \eta_k|h|^{p_k} \|^2_{ L_1^2 (M_T)}&\leq  2\Big(\frac{4^{k+1}}{L^2}+  \frac{100p_kC_{II}}{ \mu^2\sin(\alpha)} + \frac{p_k^2C_{II}^2 }{2\sin^2(\alpha)}+C_R+ \frac{1}{2}  \Big)\|   h  \|_{L^{2p_k}(B_{k-1})}^{2p_k} ,
	\end{align}

	Now we follow the usual Moser iteration proof as in \cite{cls1}. 
By Proposition \ref{coro1}, we have 
	\begin{align}\notag
	\| \eta_k |h|^{p_k}\|^2_{L^\frac{2n}{n-2}(M_T)}&\leq S(a,n)\|d(\eta_k|h|^{p_k})\|_{L_1^{2}(M_T)},
	\end{align}
so that
\begin{align}\label{moser3}
\| \eta_k |h|^{p_k}\|^2_{L^\frac{2n}{n-2}(M_T)}&\leq 
2 S(a,n)\Big(\frac{4^{k+1}}{L^2}+  \frac{100p_kC_{II}}{ \mu^2\sin(\alpha)} + \frac{p_k^2C_{II}^2 }{2\sin^2(\alpha)}+C_R + \frac{1}{2}\Big)\|   h  \|_{L^{2p_k}(B_{k-1})}^{2p_k}.
\end{align}	
	Set 
\begin{align}C_G:= \frac{100C_{II}}{ \mu^2\sin(\alpha)} + \frac{C_{II}^2 }{2\sin^2(\alpha)}+C_R+\frac{1}{2}.
\end{align}
For $p_k\geq 1$, 
	\begin{align}\label{moser4}
	\|  h \|^{2 p_k}_{L^{\frac{2n p_k }{n-2} }(B_k)}&\leq 
	2 S(a,n)p_k^24^{k+1}\Big(\frac{1}{L^2}+  C_G \Big)\|   h  \|_{L^{2p_k}(B_{k-1})}^{2p_k}.
	\end{align}
	Choose now $p_{k}=(\frac{n  }{n-2})^k.$
	Then taking $p_k$ roots of \eqref{moser4} and iterating yields 
	\begin{align}\label{moser5}
	\|  h \|^2_{L^{ 2p_{k+1}}(B_k)}&\leq 
	\prod_{j=0}^k  S(a,n)^{\frac{1}{p_j}}4^{\frac{j+2}{p_j}}p_j^{\frac{2}{p_j}}\Big(\frac{1}{L^2}+  C_G \Big)^{\frac{1}{p_j}}\| h\|^2_{L^{2 p_j}(B_{j-1})}.
	\end{align}
	Taking the limit as $k\to\infty$ and setting 
\begin{align}c_n:=\prod_{k=1}^\infty 4^{\frac{k+2}{p_k}}
\end{align} 
yields 
	\begin{align}\label{moserinfty}
	\|  h \|^2_{L^{ \infty}(B_L(p))}&\leq 
	 c_n S(a,n)^{\frac{n}{2}}\Big(\frac{1}{L^2}+  C_G \Big)^{\frac{n}{2}} \| h\|^2_{L^{2}(B_{2L}(p))}.
	\end{align} 
	If a harmonic form $h$ satisfies Dirichlet boundary conditions, then $\ast h$ satisfies Neumann boundary conditions, and $|\ast h| = |h|$. Hence the Dirichlet estimate implies the Neumann estimate. 
\end{proof}

\subsection{Non Uniformly Discrete Sequences}

Let $(X^n, g)$ be a simply connected manifold of dimension $n\geq 3$. Assume the sectional curvature is pinched:
\[
-a^2\leq\sec_g\leq-1
\]	
with $a\geq 1$. Let $(M_l, g_l)$ be a sequence of closed manifolds BS-converging to $(X, g)$, which we now assume \emph{not} to be uniformly discrete. For any element $(M_l, g_l)$ in the sequence, consider a small geodesic $\gamma$ such that
\[
l(\gamma)\leq c_1\exp(-c_2 a)\mu^n a^{n-1}
\]
as in \eqref{newconst} of Lemma \ref{BCD1}. Remove from $M_l$ the union of the modified tubes $V^l_{\gamma}$  to get a manifold with boundary
\[
M_{l,T}:= M_{l}\setminus \cup_{\gamma}V^l_{\gamma}, \quad \partial M_{l,T}=\cup_{\gamma}H^l_{\gamma}.
\]
Denote by $N_T(M_l)$ the number of disjoint tubes $V^l_{\gamma}$ in $M_l$. From the long exact sequence in cohomology,
\[
...\rightarrow H^{k}(M_{l,T}, \partial M_{l,T}; \IR)\rightarrow H^k(M_l; \IR)\rightarrow H^k(\cup_{\gamma}V^l_{\gamma}; \IR)\rightarrow H^{k+1}(M_{l,T}, \partial M_{l,T}; \IR)\rightarrow...,
\]
we  obtain the inequality
\begin{align}\label{bettiineq}
b_{k}(M_l)\leq \dim_{\IR}H^{k}(M_{l,T}, \partial M_{l,T}; \IR) + N_{T}(M_l).
\end{align}

The next lemma shows that we can control $N_T(M_l)$ in terms of the total volume.

\begin{lemma}\label{numbertubes}
	Let $(X^n, g)$ be a simply connected manifold of dimension $n\geq 3$ with
	\[
	-a^2\leq\sec_g\leq-1,
	\]	
	and $a\geq 1$. If $(M_l, g_l)$ is a sequence of closed manifolds BS-converging to $(X, g)$, then 
\begin{align}\frac{N_{T}(M_{l})}{Vol_{g_l}(M_{l})}<\frac{\rho(M_l,\mu)}{c_n\Big(\frac{21}{200}\mu\Big)^n}.
\end{align} 
\end{lemma}

\begin{proof}
	By Lemma \ref{quanta}, every Margulis tube $V^l_{\gamma}$ satisfies
	\[
	Vol(V^{l}_{\gamma})>c_n\Big(\frac{21}{200}\mu\Big)^n:=\epsilon_{0}
	\]
	where $c_n$ is a positive constant depending on the dimension only, and $\mu(n, a)>0$ is the usual Margulis constant for a negatively curved  $a$-pinched $n$-manifold. By Lemma \ref{lbound}, 
if $\gamma$ is a short geodesic satisfying Equation \eqref{newconst}, then
	\[
	V^l_{\gamma}\subset (M_{l})_{<\mu}.
	\]
	  Thus, we have
	\[
	 \frac{N_{T}(M_{l})\cdot\epsilon_{0}}{Vol(M_{l})}\leq  \frac{\sum_{\gamma}Vol(V^l_{\gamma})}{Vol(M_l)}<  \rho(M_{l},\mu).
	\]
\end{proof}

\begin{lemma}\label{lemma1}
	Let $(X^n, g)$ be a simply connected manifold of dimension $n\geq 3$ with
	\[
	-a^2\leq\sec_g\leq-1,
	\]	
	and $a\geq 1$. Let $(M_l, g_l)$ be a sequence of closed manifolds BS-converging to $(X, g)$. For any $k\in\IN$ such that 
	\[
	a_{n, k}=(n-1)-2ka> 0,
	\]
	we have for all $R\gg1$, 
	\[
	 \frac{\dim_{\IR}H^{k}(M_{l,T}, \partial M_{l,T}; \IR)}{Vol_{g_l}(M_l)}\leq  \left(\begin{array}{c}n\\k\end{array}\right)r(a,n) \rho(M_{l},R)+c(n, k)e^{-a_{n, k}(R-1)} . 
	\]
If $a_{n, k}= 0,$ we have for all $R\gg1$, 
	\[
	 \frac{\dim_{\IR}H^{k}(M_{l,T}, \partial M_{l,T}; \IR)}{Vol_{g_l}(M_l)}\leq  \left(\begin{array}{c}n\\k\end{array}\right)r(a,n) \rho(M_{l},R)+d(n, k)(R-2)^{-1} . 
	\]
\end{lemma}

\begin{proof}
	For any $l\in\IN$, denote by
	\[
	i: \partial M_{l,T}\rightarrow M_{l,T}
	\]
	the injection map. Recall that  the relative cohomology $H^{k}(M_{l,T}, \partial M_{l,T}; \IR)$ is isomorphic to the space of harmonic forms satisfying Dirichlet boundary conditions:
	\[
	H^{k}(M_{l,T}, \partial M_{l,T}; \IR)\simeq\{\alpha\in C^{\infty}(\Omega^k T^*(M_{l,T}\backslash\partial M_{l,T}))\,\, | \,\, \Delta_{k}\alpha=0,\,\, i^*\alpha=0\}.
	\]
	Let $\{\alpha_i\}_{i}$ be an $L^2$ orthonormal basis of harmonic forms satisfying Dirichlet boundary conditions, and consider the function
	\[
	TrK(x, x)=\sum_i|\alpha_{i}(x)|^2
	\]
	which satisfies
	\[
	\int_{M_{l,T}}TrK(x, x)dv = \dim_{\IR}H^k(M_{l,T}, \partial M_{l,T}; \IR).
	\]
	By Proposition \ref{bndrymoser}, for any unit norm harmonic $k$-form $\alpha$ satisfying Dirichlet boundary conditions in $M_{l,T}$, there exists a constant $r(a, n)>0$ such that
	\begin{align}\label{thinest}
	|\alpha(p)|^2\leq r(a, n),
	\end{align}
    for any  $p\in M_{l,T}.$ 
 On the other hand if we take a point in $p\in M_{l,T} $ with large injectivity radius,  we can apply Theorem \ref{pinchedprice} to obtain much stronger bounds. 
Consider $p\in M_{l}$ with $inj_g(p)=R\gg1$. We observe that any such point lies in $M_{l,T}$ and must be quite distant from $\partial M_{l,T}\subset M_l$. Indeed, by Lemma \ref{lbound}, as any point $q\in\partial M_{l,T}=\cup_{\gamma}H^l_{\gamma}$ satisfies $\inj_{g_l}(q)\leq\frac{26}{50}\mu$, we have
    \begin{align}\label{Rminus}
    d_{g_{l}}(p, q)\geq \inj_{g_{l}}(p)-\frac{26}{50}\mu>R-\mu>R-1,
    \end{align}
    as we have assumed $\mu<1$. Thus  if $p\in M_{l}$ has $\inj_{g_l}(p)\geq R>>1,$ then $\bar B_{R-1}(p)\subset  M_{l,T}^0$. Set 
    \[
    (M_{l,T})_{< R}:= (M_l)_{< R}\cap M_{l,T}
    \]
    and similarly $(M_{l,T})_{\geq R}$ as its complement in $M_{l,T}$. The Price inequality given in Theorem \ref{pinchedprice} then tells us that for $R>>1,$
    \begin{equation}\label{boundary1}
    \int_{(M_{l,T})_{\geq R}}TrK(x, x)dv\leq  \;     \begin{cases} 
    c(n, k)e^{-a_{n, k}(R-1)}Vol((M_{l,T})_{\geq R}) &\mbox{if }\quad    a_{n, k} > 0 ,  \\
    d(n, k)(R-2)^{-1}Vol((M_{l,T})_{\geq R}) & \mbox{if }\quad   a_{n, k} = 0.
    \end{cases}
    \end{equation}
    Moreover, by \eqref{runtiprun}
    \begin{align}\label{boundary2}
    \int_{(M_{l,T})_{< R}}TrK(x, x)dv\leq \left(\begin{array}{c}n\\k\end{array}\right)r(a,n) Vol((M_{l,T})_{< R}).
    \end{align}
    Combining \eqref{boundary1} and \eqref{boundary2} together with the inequalities  
    \[
     \frac{Vol_{g_l}((M_{l,T})_{<R})}{Vol_{g_l}(M_l)}\leq\rho(M_l,R),
    \]
    we immediately obtain  
    the desired inequalities.
\end{proof}


We now have the main theorem of this section.

\begin{theorem}\label{barvolume}
	Let $(X^n, g)$ be a simply connected manifold of dimension $n\geq 3$ with
	\[
	-a^2\leq\sec_g\leq-1,
	\]	
	and $a\geq 1$. Let $(M_l, g_l)$ be a sequence of closed manifolds BS-converging to $(X, g)$. For any $k\in\IN$ such that 
	\[
	a_{n, k}=(n-1)-2ka> 0,
	\]
we have for all $R\gg1$, 
	\[
	 \frac{b_k(M_{l } )}{Vol_{g_l}(M_l)}\leq  (\epsilon^{-1}_0+\left(\begin{array}{c}n\\k\end{array}\right)r(a,n) )\rho(M_{l},R)+c(n, k)e^{-a_{n, k}(R-1)} . 
	\]
If $a_{n, k}= 0,$ we have for all $R\gg1$, 
	\[
	 \frac{b_k(M_{l } )}{Vol_{g_l}(M_l)}\leq (\epsilon^{-1}_0+ \left(\begin{array}{c}n\\k\end{array}\right)r(a,n)) \rho(M_{l},R)+d(n, k)(R-2)^{-1} . 
	\]
	Consequently
	\[
	\lim_{l\rightarrow\infty}\frac{b_{k}(M_{l})}{Vol_{g_l}(M_{l})}=0.
	\]
\end{theorem}
\begin{proof}
	The proof follows from \eqref{bettiineq}, Lemma \ref{numbertubes}, and Lemma \ref{lemma1}, and the monotonicity of in $R$ of $\rho(M,R)$.
\end{proof}

	\[
	\]


\end{document}